\documentclass[12pt]{amsart}
\usepackage{amssymb}
\usepackage[mathscr]{eucal}
\usepackage{epsf}
\usepackage{epsfig}
\usepackage{color}
\DeclareGraphicsExtensions{.pstex,.eps,.epsi,.ps}
\newtheorem{theor}{~~~~Theorem}

\newtheorem{cor}{~~~~Corollary}
\newtheorem{lemma}{~~~~Lemma}

\allowdisplaybreaks
\newtheorem{remark}{~~~~Remark}
\newtheorem{defin}{~~~~Definition}

 \newcommand{\ham}{\mathbf H}
 \newcommand{\x}{\mathbf x}
 \newcommand{\p}{\mathbf p}

 \newcommand{\tv}{\mathbf v}
 \newcommand{\w}{\mathbf w}

\begin{document}
\bibliographystyle{plain}
\title[ Curvatures and hyperbolic flows in Finsler geometry]{Curvatures and hyperbolic flows for natural mechanical systems in Finsler geometry}

\author{Chengbo Li}
\address{School of Mathematics, Tianjin University, Tianjin, 300072, P.R.China}
\email{chengboli@tju.edu.cn}
\thanks{C. Li's research was supported  by the Scientific Research Foundation for the Returned Overseas Chinese Scholars, State Education Ministry. }.

\begin{abstract}
 We consider a natural mechanical system on a Finsler manifold and
study its \emph{curvature} using the intrinsic Jacobi equations (called \emph{Jacobi curves})
along the extremals of the least action of the system. The curvature for such a system is expressed in terms of the
 Riemann curvature and the Chern curvature (involving the gradient of the potential) of the Finsler manifold and the Hessian of the potential w.r.t. a Riemannian metric induced from the Finsler metric. As an application,
  we give sufficient
conditions for the Hamiltonian flows of  the least action  to be hyperbolic and show
  new examples of Anosov flows.

\end{abstract}

\date{\today}

\maketitle

\section{Introduction}
In 1990s A. Agrachev and R. Gamkrelidze(\cite{AgGa})
 proposed the program of studying  an extremal of the optimal control problems  on a manifold $M$ through the intrinsic Jacobi equations (called \emph{Jacobi curves}) along the extremal. The Jacobi curve is a curve in a Lagrangian Grassmannian defined up to a symplectic transformation and containing all information about the solutions of the Jacobi equations along this extremal. Based on the study of the differential geometry of the parameterized curves in Lagrangian Grassmannians(see \cite{AgGa, LiZe1,LiZe2} and the recent monograph \cite{AgBoBa}), we can apply  to Jacobi curves to constructing the curvature-type invariants (called curvatures in short) for natural mechanical systems on various smooth geometric structures including (sub-)Riemannian and (sub-)Finsler manifolds. The curvatures for (sub-)Riemannian and (sub-)Finsler geodesic problems then can be taken as particular cases. We refer the readers to \cite{ABR1,ABR2} for the development of sub-Riemannian curvatures in the last few years.  Moreover, by using the curvatures we can  derive the qualitative properties of the extremals of the optimal control problems, such as various comparison theorems and hyperbolicity.

\smallskip




 For the simplest case of a Riemannian geodesic problem, the aforementioned curvature invariants  essentially coincides with the Riemannian curvature tensor (\cite{AgGa}). Further, for the least action problems of  a natural mechanical system  on a Riemannian manifold, the curvatures are expressed by the Riemannian curvature tensor and the Hessian of the potential (see \cite{AgChZe}).

%

 It is very natural to expect that for a Finsler geodesic problem  the curvature coincides with the Riemann curvature of the Finsler manifold(at some reference vector) and it is verified from a unified Hamiltonian viewpoint (\cite{Le1}). Now a natural question arises: how to express the curvature for a natural mechanical system on a Finsler manifold using the geometric quantities in the Finsler manifold and the potential?

%

\medskip
One reason to find the answer to the above question is that it can be used to studying  the hyperbolicity of the Hamiltonian flows. We first of all recall the following  (see e.g. \cite{KaHa})

\begin{defin}
Let $e^{tX},t\in\mathbb R$ be the flow generated by the vector
field $X$ on a manifold $P$. A compact invariant set $A\subset P$ of
the flow $e^{tX}$ is called a hyperbolic set if there exists a
Riemannian structure in a neighborhood of $A$, a positive constant
$\delta$, and a splitting: $T_zP=E_z^+\oplus E_z^-\oplus\mathbb
RX(z),\ z\in A$ such that $X(z)\neq 0$ and
\begin{enumerate}
\item $e^{tX}_*E^+_z=E^+_{e^{tX}z},\ e^{tX}_*E^-_z=E^-_{e^{tX}z},$
\item $\|e^{tX}_*\zeta^+\|\geq e^{\delta t}\|\zeta^+\|,\
\forall t>0, \forall \zeta^+\in E^+_z,$
\item $\|e^{tX}_*\zeta^-\|\leq e^{-\delta t}\|\zeta^-\|,\
\forall t>0, \forall \zeta^-\in E^-_z.$
\end{enumerate}
If the entire manifold $P$ is a hyperbolic set, then the flow
$e^{tX}$ is called a flow of Anosov type.
\end{defin}

It is well known that the geodesic flows on a closed Riemannian
manifold with negative sectional curvature is of  Anosov type
(\cite{Ano}). Such a result has a Finsler version (\cite{Fou}): the geodesic flows on a closed reversible Finsler manifold with negative flag curvature must be of Anosov type. Actually both of them can be  derived from a more general criteria of hyperbolic flows from the Hamiltonian viewpoint (\cite{AgCh}) and this criteria can also be used to get  some sufficient conditions for  the  Hamiltonian flows for a natural mechanical system  on a Finsler manifold.

\medskip
The purpose of the present draft is two-fold. On one hand we interpret the curvatures for the Finsler least action problems by the curvature tensors (Riemannian curvature and Chern curvature) and the Hessian of the potential in Finsler geometry by the formulas which can be plugged into the framework of calculations in \cite{LiZe3}. On the other hand we also get the sufficient conditions for the Hamiltonian flows for Finsler least action problems to be of Anosov type on Finsler manifolds.

\medskip
 Note that we always use  Einstein summation convention:  when an index variable appears twice in a single term it implies summation.

\section{Main results}


In this section we present the main results on curvatures and hyperbolic flows. The proofs are postponed to the next section.
\subsection{Preliminaries in Finsler geometry}

We first recall various notations which are needed in the rest of the draft (see e.g. \cite{Shen} for more details).

Given a local coordinate $(x^i)^n_{i=1}$ on an open set $\Omega$ in a smooth manifold $M$ of dimension $n$, we will always use the coordinate $(x^i,y^j)^n_{i,j=1}$ of $T\Omega$ with
$$\tv=y^j\frac{\partial}{\partial x^j}|_{\x} \in T_{\x}M \quad \quad \mbox{for }\x \in \Omega.$$
While we use the  coordinate $(x^i,p_j)^n_{i,j=1}$ of the cotangent bundle $T^*\Omega$ with
$$\p=p_j dx^j|_{\x} \in T^*_{\x}M \quad \quad \mbox{for }\x \in \Omega.$$

\begin{defin}
{\bf(Finsler structures)}
 A nonnegative function $F:TM\rightarrow [0,\infty)$ is called a smooth Finsler structure of $M$ if the following three conditions hold.\\
(1) (Regularity) $F$ is smooth on $TM\backslash0$, where $0$ stands for the zero section.\\
(2) (Positive 1-homogeneity)  $F(c\tv)=cF(\tv)$ for all $\tv\in TM$ and $c>0$.\\
(3) (Strong convexity) The $n\times n$ matrix
$$(g_{ij}(\tv))^n_{i,j=1}:=\left(\frac{1}{2}\frac{\partial^2(F^2)}{\partial y^i \partial y^j}(\tv)\right)^n_{i,j=1}$$
is positive-definite for all $\tv\in TM\backslash0$.\\
We call the manifold $M$ a (smooth) Finsler manifold with Minkowski norm $F$.
\end{defin}
For $\x_0,\x_1 \in M$, we define the \emph{distance} from $\x_0$ to $\x_1$ in a natural way by
$$d(\x_0,\x_1):=\inf_{\x}\int^1_0F(\dot\x(t))dt,$$
where the infimum is taken over all $C^1$-curves $\x:[0,1]\rightarrow M$ such that $\x(0)=\x_0$ and $\x(1)=\x_1$. Since $F$ is only positively homogeneous,  the distance can be non-reversible, i.e. $d(\x_0,\x_1)\neq d(\x_1,\x_0)$ for some $\x_0,\x_1\in M$. A smooth curve $\x(\cdot)$ on $M$ is called a $\emph{geodesic}$ if it is locally minimizing and has a constant speed (i.e. $F(\dot\x)$ is constant).

For each $\tv\in T_{\x}M\backslash\{0\}$, the positive-definite matrix $(g_{ij}(\tv))^n_{i,j=1}$  induces the Riemannian structure $g_{\tv}$ of $T_{\x}M$ as
$$g_{\tv}\left(a^i\frac{\partial}{\partial x^i}|_{\x},b^j\frac{\partial}{\partial x^j}|_{\x}\right):=a^ib^jg_{ij}(\tv).$$


For later convenience, we recall a basic fact on homogeneous functions.

\begin{theor}\label{homo}
 Suppose that a differentiable function $H:\mathbb{R}^n\backslash\{0\}\rightarrow \mathbb{R}$ is positively r-homogeneous, i.e. $H(c\tv)=c^rH(\tv)$ for some $r\in \mathbb{R}$ and all $c>0$
and $\tv\in \mathbb{R}^n\backslash\{0\}$. Then we have
\begin{center}
$\frac{\partial H}{\partial y^i}(\tv)y^i=rH(\tv)$ \quad \quad for all $\tv\in \mathbb{R}^n\backslash\{0\}$.
\end{center}
\end{theor}

\medskip
The $\emph{Cartan tensor}$
$$C_{ijk}(\tv):=\frac{1}{2}\frac{\partial g_{ij}}{\partial y^k}(\tv)\quad\quad \mbox{for } \tv\in TM\backslash0$$
is a pure Finsler quantity. Indeed, ${C_{ijk}}'s$ vanish everywhere on $TM\backslash0$ if and only if $F$ comes from a Riemannian metric. As $g_{ij}$ is positively 0-homogeneous on each $T_{\x}M\backslash0$, Theorem \ref{homo} yields
\begin{equation}\label{cartan}
C_{ijk}(\tv)y^i=C_{ijk}(\tv)y^j=C_{ijk}(\tv)y^k=0
\end{equation}
for all $\tv\in TM\backslash0$.

\smallskip
 Define the formal Christoffel symbols
\begin{equation*}\label{chris}
\gamma^k_{ij}(\tv):=\frac{1}{2}g^{kl}(\tv)\left\{\frac{\partial g_{il}}{\partial x^j}(\tv)+\frac{\partial g_{lj}}{\partial x^i}(\tv)-\frac{\partial g_{ij}}{\partial x^l}(\tv)\right\}\  \hbox{for}\ \tv\in TM\backslash 0,
\end{equation*}
where $(g^{ij}(\tv))$ stands for the inverse matrix of $(g_{ij}(\tv))$. We also introduce the geodesic spray coefficients and the nonlinear connection. Let
$$G^i(\tv):=\gamma^i_{jk}(\tv)y^jy^k,\quad N^i_j(\tv):=\frac{1}{2}\frac{\partial G^i}{\partial y^j}(\tv) \  \hbox{for}\ \tv\in TM\backslash 0,$$
and $G^i(0)=N^i_j(0)=0$ by convention.

Chern connection is torsion free and almost compatible with the metric and its coefficients  are given by
$$\Gamma^i_{jk}:=\gamma^i_{jk}-g^{il}(C_{jlm}N^m_k+C_{lkm}N^m_j-C_{jkm}N_l^m)\ \hbox{on}\  TM\backslash 0.$$
One can show
\begin{equation}\label{N_Gamma}
N^i_j=\Gamma^i_{jk}y^k.
\end{equation}

\smallskip
We recall various curvature tensors in Finsler geometry which will be needed in the draft. Let $U=U^k\frac{\partial}{\partial{x^k}}, V=V^l\frac{\partial}{\partial{x^l}}, W=W^j\frac{\partial}{\partial{x^j}}$. The Chern connection gives the \emph{Riemannian curvature tensor} $\mathcal R$ which can be written as
\begin{equation*}
\mathcal R(U,V)W:=R^{\ i}_{j\ kl}U^kV^lW^j\frac{\partial}{\partial{x^i}}
\end{equation*}
where
\begin{equation}\label{RieCoe}
R^{\ i}_{j\ kl}=\frac{\partial \Gamma^i_{jl}}{\partial x^k}-\frac{\partial\Gamma^i_{jk}}{\partial x^l}+\frac{\partial \Gamma^i_{jk}}{\partial y^m}N^m_l-\frac{\partial \Gamma^i_{jl}}{\partial y^m}N^m_k+\Gamma_{jl}^m\Gamma_{mk}^i-\Gamma_{jk}^m\Gamma_{ml}^i.
\end{equation}

Another curvature is \emph{Chern curvature} which is a non-Riemannian curvature defined by
\begin{equation*}\label{CheCoe}
\mathbf P_{\tv}(U,V,W):=P^{\ i}_{j\ kl}(\tv)U^kV^lW^j\frac{\partial}{\partial x^i},
\end{equation*}
where $\mathbf P^{\ i}_{j\ kl}=-\frac{\partial \Gamma^i_{jk}}{\partial y^l}.$

\medskip

Let $P\subset T_xM$ be a tangent plane. For a vector $\tv\in P\backslash \{0\}$, define
\begin{equation}\label{FlaCur}
K(P,\tv):=\frac{g_{\tv}({\mathcal R(\w,\tv)\tv,\w})}{g_{\tv}(\tv,\tv)g_{\tv}(\w,\w)-g_{\tv}(\tv,\w)^2},
\end{equation}
where $\w \in P$ such that $P = {\rm span}\{\tv,\w\}.$ The number $K(P,\tv)$ is called the \emph{flag curvature} of the flag $(P,\tv)$ in $T_\x M$.

\medskip
Finally we recall the Legendre transform 
 in Finsler geometry.
 Denote by $F^*$ the dual norm on $T^*M$, i.e.
 $$F^*(\p):=\sup_{F(\tv)=1}\ \p(\tv), \ \p\in T^*_{\x}M.$$
Recall that we can write
$$F^*(\p)^2=g^*_{ji}(\p)p_ip_j,\ g^*_{ji}(\p)=\frac{1}{2}\frac{\partial^2((F^*)^2)}{\partial p_j\partial p_i}(\p).$$

Let us denote by $\mathcal L^*:T^*M\rightarrow TM$ the \emph{Legendre transform} associated with $F$ and $F^*$. More precisely,
$\mathcal L^*(\p)$ is the unique vector $\tv=y^i(\p)\frac{\partial}{\partial{x_i}}\in T_{\x}M$ such that
\begin{equation*}\label{legend}
\p(\tv)=F^*(\p)^2,\quad F(\tv)=F^*(\p).
\end{equation*}
For later use, we recall the following relations.
\begin{equation}\label{identi}
y^i(\p)=g^*_{ji}(\p)p_j,\quad p_i=g_{ij}(\tv)y^j(\p).
\end{equation}
%

\medskip

In the remainder of the draft, for the reason of simplicity we adopt the following convention on the notations: denote by $\tv$ the image of $\p$ via the Legendre transform $\mathcal L^*$, i.e. $\tv=\mathcal L^*(\p)$ and write $\tv=y^i\frac{\partial}{\partial x_i}.$

\subsection{Curvatures for least action problems for a natural mechanical system on a Finsler manifold}

On a Finsler manifold $M$ with Minkowski  norm $F$, consider the Finsler version of the least action problem of a natural mechanical system
\begin{equation}\label{FinLeast}
\begin{split}
&A(\x(\cdot))=\int^T_0 \frac{1}{2}F(\dot \x(t))^2-U(\x(t))dt\rightarrow {\rm min}\\
&\x(0)=\x_0,\quad \x(T)=\x_1.
\end{split}
\end{equation}

As in Riemannian geometry, the minimizers coincide when minimizing the length and the kinetic energy. Hence, when the potential $U$ is identical to a null function, the above problem reduces to a Finsler geodesic problem.  And since a Riemannian metric is a Finsler metric satisfying quadratic condition, the Riemannian least action problem is a particular case of the Finsler least action problem.

As optimal control problems, the Finsler least action problems can be solved by Pontryagin Maximum Principle(\cite{Pon}). Let $\sigma$ be the canonical symplectic form on $T^*M$, i.e. $\sigma=dx^i\wedge dp_i$. Let $\ham$ be the maximized Hamiltonian(see Lemma \ref{hamil} below). Then Pontryagin Maximum Principle tells that
 the minimizers are projections to the manifold $M$ of the Hamiltonian flows generated by the vector field $\vec\ham$ on $T^*M$.

\begin{lemma}\label{hamil} The maximized Hamiltonian $\ham$ is written as follows.
\begin{equation}\label{HamMec}
\ham(\p)=\frac{1}{2}F^*(\p)^2+U(\x)=\frac{1}{2}F(\tv)^2+U(\x).
\end{equation}
\end{lemma}
\begin{proof}As a result of Theorem \ref{homo}, $F^2(\tv)=g_{ij}y^iy^j$, hence
\begin{equation*}
\begin{split}
\ham(\p)&=\max_{\tv\in T_{\x}M}(\p(\tv)-\frac{1}{2}F^2(\tv)+U(\x))
\\&=\max_{\tv\in T_{\x}M}(\p(\tv)-\frac{1}{2}g_{ij}y^iy^j+U(\x))\\
&=\frac{1}{2}F^*(\p)^2+U(\x)=\frac{1}{2}F(\tv)^2+U(\x).
\end{split}
\end{equation*}
\end{proof}

We will construct the curvatures for Finsler least action problem.
 For this let us introduce the Jacobi curves associated with an extremal of the Finsler least action problem to describe its dynamical property.
 Let us fix the level set of the Hamiltonian function $\ham$:
$$\mathcal{H}_c:=\{\lambda\in T^*M| \ham(\lambda)=c\},\ c>0.$$
Let $\Pi_{\lambda}$ be the vertical subspace of
$T_{\lambda}\mathcal{H}_c$, i.e.
$$
\Pi_{\lambda}=\{\xi\in T_{\lambda}\mathcal{H}_c:
\pi_*(\xi)=0\},
$$
where $\pi: T^*M\longrightarrow M$ is the canonical projection.
The curve defined by
\begin{equation}\label{Jacobi}
t\longmapsto \mathfrak J_{\lambda}(t):=e_*^{-t\vec
\ham}(\Pi_{e^{t\vec \ham}\lambda})/\{\mathbb{R}\vec \ham(\lambda)\}.
\end{equation}
is called the \emph{Jacobi curve} of the extremal $e^{t\vec \ham}\lambda$ (attached at the point $\lambda$).
The curve $\mathfrak J_\lambda(t)$ is a curve in the Lagrange Grassmannian of the linear symplectic space
$W_\lambda = T_\lambda\mathcal H_c/{\mathbb R\vec \ham(\lambda)}$ (endowed with the symplectic form induced in the obvious way by the canonical symplectic form $\sigma$ of $T^*M$).

Next we introduce another version of Jacobi curves $\bar{\mathfrak J}_\lambda(\cdot)$, called \emph{non-reduced Jacobi curves}, by
\begin{equation}\label{Jacobi2}
t\longmapsto \bar{\mathfrak J}_{\lambda}(t):=e_*^{-t\vec
\ham}(T_{e^{t\vec \ham}\lambda}T^*_{\pi(e^{t\vec \ham}\lambda)}M).
\end{equation}
There is a close relation between the two kinds of Jacobi curves: $\mathfrak J_\lambda(\cdot)$ can be obtained from  $\bar{\mathfrak J}_\lambda(\cdot)$  after the reduction of the first integral $\ham$ (of the Hamiltonian flow generated by $\ham$).

\medskip
Next we give a concise description of the construction of the curvature-type invariants for the parametrized curves in some Lagrangian Grassmannians. For our purpose we focus on the case of regular curves and refer the reader to the relevant references for more general cases.

Recall that the tangent space $T_\Lambda L(W)$ to the
Lagrangian Grassmannian $L(W)$ of a linear symplectic space $W$
(endowed with a symplectic form $\omega$) at the point $\Lambda$ can
be naturally identified with the space ${\rm Quad}(\Lambda)$ of all
quadratic forms on linear space $\Lambda\subset W$.
Namely, given $\mathfrak V\in T_\Lambda L(W)$ take a curve
$\Lambda(t)\in L(W)$ with $\Lambda(0)=\Lambda$ and
$\dot\Lambda=\mathfrak V$. Given some vector $l\in\Lambda$, take a
curve $\ell(\cdot)$ in $W$ such that $\ell(t)\in \Lambda(t)$ for all
$t$ and $\ell(0)=l$. Define the quadratic form
\begin{equation}
\label{quad} Q_{\mathfrak V}(l)=\omega(l,\frac{d}{dt}\ell(0)).
\end{equation}
Using the fact that the spaces $\Lambda(t)$ are Lagrangian,
it is easy to see that $Q_{\mathfrak V}(l)$ does not depend on the
choice of the curves $\ell(\cdot)$  and $\Lambda(\cdot)$ with the
above properties, but depends only on $\mathfrak V$.
So, we have the linear mapping from $T_\Lambda L(W)$ to the spaces
${\rm Quad}(\Lambda)$, $\mathfrak V\mapsto Q_{\mathfrak V}$.
A simple counting of dimensions shows that this mapping is a
bijection and it defines the required identification. A curve
$\Lambda(\cdot)$ in a Lagrangian Grassmannian is called
\emph{regular}, if its velocity  is a non-degenerate quadratic form
at any time $t$. A curve $\Lambda(\cdot)$ is called \emph{monotone}
(monotonically nondecreasing or monotonically nonincreasing) if the
velocity is sign definite (nonnegative or nonpositive) at any point.

    Note that one can show that either Jacobi curves $\mathfrak J_\lambda(\cdot)$ or the non-reduced Jacobi curves $\bar{\mathfrak J}_\lambda(\cdot)$  for the case of Finsler least action is regular and monotone (see, for example, \cite[Proposition 1]{AgZe1}).

The curvatures for regular curves in Lagrangian
Grassmannians are constructed in earlier work \cite{AgGa} and can be taken as a particular case of the results in \cite{LiZe1, LiZe2}.

\begin{theor}\label{structure}
Let $\Lambda(\cdot)$ be a regular curve in the Lagrangian Grassmannian
$L(W)$ of a $2n$-dimensional linear symplectic space $W$. Then there
exists a moving Darboux frame $(E(t), F(t))$ of $W$:
$$E(t)=(e_1(t),...,e_{n}(t)),\ F(t)=(f_1(t),...,f_{n}(t))$$
such that $\Lambda(t)={\rm span}\{E(t)\}$ and there exists a
one-parametric family of symmetric matrices $ R(t):
\Lambda(t)\rightarrow \Lambda(t)$ satisfying
\begin{equation}
\label{3structeq}
\begin{cases}
 E^{\prime}(t)=F(t),\\
 F^{\prime}(t)=-E(t)R(t).
\end{cases}
\end{equation}
The moving frame $(E(t), F(t))$ is a called a normal moving
frame of $\Lambda(t)$. A moving frame $(\widetilde E(t),\widetilde
F(t))$ is a normal moving frame of $\Lambda(t)$ if and only there
exists a constant orthogonal matrix $O$ of size $n\times n$ such
that
\begin{equation}\label{3U}
\widetilde E(t)=E(t)O,\ \widetilde F(t)=F(t)O.
\end{equation}
\end{theor}

As a matter of fact, normal moving frames define a principal
$O(n)$-bundle of symplectic frame in $W$ endowed with a canonical
connection. Also, relations \eqref{3U} imply that the following
$n$-dimensional subspaces
\begin{equation}\label{3V}
\Lambda^{\rm trans}(t)={\rm span}\{F(t)\}={\rm span}\{f_1(t),...,f_n(t)\}
\end{equation}
of $W$ does not depend on the choice of the normal moving frame. It
is called the \emph{canonical complement} of $\Lambda(t)$ in $W$.
Moreover, the subspaces $\Lambda(t)$ and $\Lambda^{\rm trans}(t)$
are endowed with the \emph {canonical Euclidean structure} such that
the tuple of vectors $E(t)$ and $F(t)$ constitute an orthonormal
frame w.r.t. to it, respectively.
The linear map from $\Lambda(t)$ to $\Lambda(t)$ with the matrix
$R(t)$ from (\ref{3structeq}) in the basis $\{E(t)\}$, is
independent of the choice of normal moving frames and is
self-adjoint with respect to the Euclidean structure in
$\Lambda(t)$. It will be denoted by $\mathfrak R(t)$ and it is
called the \emph {curvature map} of the curve $\Lambda(t)$.

\medskip
The construction of curvature map for curves in Lagrangian
Grassmannians naturally applies to the Jacobi curves $\mathfrak J_\lambda(\cdot)$ and non-reduced Jacobi curves $\bar{\mathfrak J}(\cdot)$ under consideration.

Let $\mathfrak J^{trans}(t)$ and $\bar{\mathfrak J}^{trans}(t)$ be the canonical complement of $\mathfrak J(t)$ (in the linear symplectic space $W_\lambda$) and $\bar{\mathfrak J}(t)$ (in the linear symplectic space $T_\lambda T^*M$), respectively. Then, $\mathfrak J^{trans}_\lambda:=\mathfrak J^{trans}(0)$ and $\bar{\mathfrak J}^{trans}_\lambda:=\bar{\mathfrak J}^{trans}(0)$ give the canonical complement of $ \Pi_\lambda$ (in $W_\lambda$) and $T_\lambda T^*_\x M$ (in  $T_\lambda T^*M$), respectively. Note here we used that $\mathfrak J(0)$ is naturally identified with $\Pi_\lambda$ and $\bar{\mathfrak J}(0)=T_\lambda T^*_\x M.$ See Subsection 3.2 for a more detailed discussion on the canonical complements.

Let $\mathfrak R_{\lambda}(t)$ be the curvature for the
Jacobi curve $\mathfrak J_{\lambda}(\cdot)$ and let $\bar{\mathfrak R}_{\lambda}(\cdot)$ be the curvature  for the
non-reduced Jacobi curve $\bar{\mathfrak J}_{\lambda}(\cdot)$. Then the linear maps
$$\mathfrak R_\lambda:=\mathfrak
R_\lambda(t)|_{t=0}:\Pi_\lambda\rightarrow\Pi_\lambda$$
and
$$\bar{\mathfrak R}_\lambda:=\bar{\mathfrak
R}_\lambda(t)|_{t=0}:T_{\lambda}T^*_\x M\rightarrow T_{\lambda}T^*_\x M$$
are said to be the
\emph{curvature} (at $\lambda$) and \emph{non-reduced curvature} (at $\lambda$) of the Finsler least action problem.

\subsection{Statements of the Main Results}
To show the results on curvatures we need the following notations. Let, as before, $\tv=\mathcal L^*(\p)$
 and $g_{\tv}$ be the Riemannian metric induced by the vector $\tv$. Note that $T_\lambda T_\x^*M$ is identified with $T_\x^*M$. Then for any $\xi\in T_\lambda T^*_\x M(\sim T_\x^*M)$  we associate a vector $\xi^h\in T_\x M$ via the Riemannian metric $g_{\tv}$, i.e.
$ g_{\tv}(\xi^h,\cdot)=\xi.$
 In particular  $\p^h=\tv$, which is a consequence of the identity $$(\mathcal L^*)^{-1}(\tv)=g_{\tv}(\tv,\cdot).$$

Note that $\Pi_\lambda$ is embedded in $T_\lambda T^*_\x M$. Hence, the above operator of superscript $^h$ also applies to $\xi\in \Pi_\lambda$ to get a vector $\xi^h\in T_{\x} M$.\\

 Let $\mathcal R$ be the Riemannian curvature tensor, $\mathbf P$ the Chern curvature tensor, $\mathbf{Hess}_{\tv}$ the Hessian and $\nabla_\tv U$  the gradient w.r.t. the Riemannian metric $g_{\tv}$.
\begin{theor}\label{CurFin}
The non-reduced curvature for the Finsler least action problems
 satisfies  for $\forall {\bar\xi},{\bar\eta}\in T_{\lambda}T^*_{\x} M$,
\begin{equation*}
\begin{split}
g_{\tv}((\bar{\mathfrak R}_\lambda{\bar\xi})^h,{\bar\eta}^h)&=g_{\tv}(\mathcal R({\bar\xi}^h,\tv){\tv,\bar\eta}^h)
+{\mathbf{Hess_{\tv}}}\ U({\bar\xi}^h,{\bar\eta}^h)\\
&+g_{\tv}(\mathbf{P_{\tv}}\left({\bar\xi}^h,\nabla_\tv U,{\bar\eta}^h\right),\tv).
\end{split}
\end{equation*}

\end{theor}

\begin{theor}\label{FinMec2}
The curvature $\mathfrak R_\lambda$ satisfy for $\forall \xi,\eta\in\Pi_\lambda$,
\begin{equation*}
\begin{split}
g_{\tv}(({\mathfrak R}_\lambda\xi)^h,\eta^h)&=g_{\tv}(\mathcal R({\xi}^h,\tv){\tv,\eta}^h)
+{\mathbf{Hess_{\tv}}}\ U(\xi^h,\eta^h)\\
&+g_{\tv}(\mathbf{P_{\tv}}\left(\xi^h,\nabla_\tv U,\eta^h\right),\tv)\\
& +\frac{3}{F(\tv)^2}g_{\tv}( \xi^h,\nabla_\tv U)g_{\tv}( \eta^h,\nabla_\tv U).
\end{split}\end{equation*}
\end{theor}

\begin{cor}\label{FinGeo}
For Finsler geodesic problems ($U$=0), the curvatures satisfy $\forall \xi,\eta\in\Pi_\lambda$
\begin{equation*}
g_\tv(({\mathfrak R}_\lambda\xi)^h,\eta^h)=g_{\tv}(\mathcal R(\tv,\xi^h)\eta^h,\tv).
\end{equation*}
\end{cor}

\medskip
Now we turn to the study of Anosov flows on Finsler manifolds. In the present setting, the criteria for hyperbolicity and Anosov flows in \cite{AgCh} is written as the following

\begin{theor}\label{main3}
 Let $c$ be a positive constant. Let
$S$ be a compact invariant set of the flow $e^{t\vec  \ham}$ contained
in a fixed level of $\ham^{-1}(c)$. If the curvature satisfies that
$g_{\tv}(({\mathfrak R}_\lambda\xi)^h,\xi^h)<0,\ \forall \tv,\xi\in {\mathcal H}_c$
 at every point $\x$ of $S$, then $S$ is a hyperbolic set of
the flow $e^{t\vec \ham}|_{\ham^{-1}(c)}$.
\end{theor}%

Combining this theorem and Theorem \ref{FinMec2} we have

\begin{theor}\label{main4}
Assume that the flag curvature of a closed reversible Finsler manifold $(M,F)$ is bounded  from above by $k$.
If the constant $c$ satisfies that
\begin{eqnarray*}
&~&\max_{\tv\perp \w, F(\tv)=F(\w)=2(c-U)}\{{\mathbf{Hess_{\tv}}}\ U(\w,\w)
+g_{\tv}(\mathbf{P_{\tv}}\left(\w,\nabla_\tv U,\w\right),\tv)\\
 &+&\frac{3}{4(c-U)^2}g_{\tv}( \w,\nabla_\tv U)^2\}<-4k(c-U)^2,
\end{eqnarray*}
then the flow $e^{t\vec \ham}|_{\mathcal H_c}$
is an Anosov flow. 
\end{theor}

When specializing to a closed Riemannian manifold $(M,g)$, we have the following
\begin{cor}
Assume that the sectional curvature of $(M,g)$ is bounded  from above by $k$.
If the constant $c$ satisfies that
\begin{eqnarray*}
&~&\max_{\tv\perp \w, |\tv|=|\w|=1}{\frac{\mathbf{Hess}\ U(\w,\w)}{2(c-U)}
+\frac{3}{4(c-U)^2}g( \w,\nabla U)^2}<-k,
\end{eqnarray*}
then the flow $e^{t\vec \ham}|_{\mathcal H_c}$
is an Anosov flow. 
If denote by $\|\mathbf{Hess}\ U\|$ the operator norm of $\mathbf{Hess}\ U$ and $\|\nabla U\|$ the norm operator of $g(\nabla U,\cdot)$, then above condition can be written as
$$\max_{\x\in M}\{\frac{\|\mathbf{Hess}\ U\|}{2(c-U)}
+\frac{3}{4(c-U)^2}\|\nabla U\|^2\}<-k.$$
\end{cor}
It follows immediately that the geodesic flows on a closed reversible Finsler manifold with negative flag curvature are of Anosov type.

\begin{remark}
There are partial results on the expressions of the curvatures for the least action problems of a natural mechanical system on a  contact sub-Riemannian manifold with transverse symmetries (see \cite{Li1}). And the hyperbolicity of the reduced Hamiltonian flows (reduced by the first integral from the transverse symmetries) for the least action problems of a natural mechanical system on a sub-Riemannian manifold with commutative transverse symmetries are discussed in \cite{Li2}.
\end{remark}

\section{Proofs of the main results}
In this section we show the proofs of Theorems \ref{CurFin} and \ref{FinMec2}.

\medskip
As before, let $\tv=y^i\frac{\partial}{\partial x_i}=\mathcal L^*(\p)$. Note that, for the reason of simplicity we will not write $\p$ and $\tv$ in the tensors. For example, we write $g_{ij}$ and $g^*_{ji}$ instead of $g_{ij}(\tv)$ and $g^*_{ji}(\p)$, respectively. However, one should understand that such tensors in general depend on $\p$ or $\tv$, which is the essential non-Riemannian phenomenon.

Note that  for the rest of the draft, we  consider $y^i$ as a function
of $p_1, ..., p_n$  via the Legendre transformation $\mathcal L^*$. For simplicity again, we write $y^i$ instead of $y^i(\p)$.

\subsection{Some  useful lemmas}

%
First of all, Theorem \ref{homo} implies
\begin{lemma}\label{der}
$\frac{\partial y^i}{\partial p_j}=g^*_{ij}.$
\end{lemma}

Further, let $\nabla^{\tv}_{\w}$ be the horizontal lift of $\w\in T_{\x}M$ via the Chern connection with the reference vector $\tv$. In local coordinates,
\begin{equation}\label{lift}
\nabla^{\tv}_{\frac{\partial}{\partial x^i}}=\frac{\partial}{\partial x^i}+\Gamma^k_{ij}p_k\frac{\partial}{\partial p_j}.
\end{equation}

Now we expression the Hamiltonian vector field $\vec\ham$ using the Chern connection. Note that it is also a  consequence of homogeneity on the fibres of $F^*$ (see \cite{AgGa} for this point).
\begin{lemma}\label{HamLif}
$\vec\ham(\p)=\nabla^{\tv}_{\tv}+\vec U$.
\end{lemma}
\begin{proof}
It follows from Lemma \ref{hamil} that $$\ham(\p)=\frac{1}{2}g^*_{ij}p_ip_j+U.$$
Hence,
\begin{equation*}
\begin{split}
\vec\ham(\p)&=\frac{\partial \ham}{\partial p_i}(\p)\frac{\partial}{\partial x^i}-\frac{\partial \ham}{\partial x^i}(\p)\frac{\partial}{\partial p_i}+\vec U\\
&=g^*_{ij}p_i\frac{\partial}{\partial x^j}+\frac{1}{2}\frac{\partial g^*_{ij}}{\partial p_k}p_ip_j\frac{\partial}{\partial x^k}-\frac{1}{2}\frac{\partial g^*_{ij}}{\partial x^k}p_ip_j\frac{\partial}{\partial p_k}+\vec U
\end{split}
\end{equation*}
As $g^*_{ij}$ is positively 0-homogeneous in $\p$, Theorem \ref{homo} implies
$$\frac{\partial g^*_{ij}}{\partial p_k}p_ip_j=0.$$
Therefore,
\begin{equation}\label{hamil2}
\vec\ham(\p)=g^*_{ij}p_i\frac{\partial}{\partial x^j}-\frac{1}{2}\frac{\partial g^*_{ij}}{\partial x^k}p_ip_j\frac{\partial}{\partial p_k}+\vec U.
\end{equation}

On the other hand, using \eqref{lift}, \eqref{cartan} and \eqref{identi} we have
\begin{equation*}
\begin{split}
\nabla^{\tv}_{\tv}&=y^i\frac{\partial}{\partial x^i}+\Gamma^k_{ij}p_ky^i\frac{\partial}{\partial p_j}\\
&=y^i\frac{\partial}{\partial x^i}+\left(\gamma^k_{ij}
-g^{kl}\Big(C_{ilm}N^m_j+C_{ljm}N^m_i-C_{ijm}N^m_l\Big)\right)p_ky^i\frac{\partial}{\partial p_j}\\
&=y^i\frac{\partial}{\partial x^i}+\gamma^k_{ij}p_ky^i\frac{\partial}{\partial p_j}-\Big(C_{ilm}N^m_j+C_{ljm}N^m_i-C_{ijm}N^m_l\Big)y^ly^i\frac{\partial}{\partial p_j}\\
&=y^i\frac{\partial}{\partial x^i}+\gamma^k_{ij}p_ky^i\frac{\partial}{\partial p_j}.
\end{split}
\end{equation*}

The rest are actually the one for Riemannian case. Indeed, combining Lemma \ref{der} with \eqref{chris} and \eqref{identi} we have
\begin{equation*}
\begin{split}
&(\nabla^{\tv}_{\tv}+\vec U)-\vec\ham(\p)=\gamma^k_{ij}p_ky^i\frac{\partial}{\partial p_j}+\frac{1}{2}\frac{\partial g^*_{ij}}{\partial x^k}p_ip_j\frac{\partial}{\partial p_k}\\
&=\frac{1}{2}\left\{\frac{\partial g_{il}}{\partial x^j}+\frac{\partial g_{lj}}{\partial x^i}-\frac{\partial g_{ij}}{\partial x^l}\right\}y^ly^i\frac{\partial}{\partial p_j}+\frac{1}{2}\frac{\partial g^*_{ij}}{\partial x^k}g_{is}y^sg_{jt}y^t\frac{\partial}{\partial p_k}\\
&=\frac{1}{2}\frac{\partial g_{il}}{\partial x^j}y^ly^i\frac{\partial}{\partial p_j}
+\frac{1}{2}\frac{\partial g^*_{ij}}{\partial x^k}g_{is}y^sg_{jt}y^t\frac{\partial}{\partial p_k}.
\end{split}
\end{equation*}
Using that $(g^*_{ij})$ is the inverse matrix of $(g_{ij})$,  we conclude
\begin{equation*}
\begin{split}
~&(\nabla^{\tv}_{\tv}+\vec U)-\vec\ham(\p)\\
&=\frac{1}{2}\frac{\partial g_{il}}{\partial x^j}y^ly^i\frac{\partial}{\partial p_j}
-\frac{1}{2}\frac{\partial g_{is}}{\partial x^k}g^*_{ij}y^sg_{jt}y^t\frac{\partial}{\partial p_k}\\
&=\frac{1}{2}\frac{\partial g_{il}}{\partial x^j}y^ly^i\frac{\partial}{\partial p_j}
-\frac{1}{2}\frac{\partial g_{is}}{\partial x^k}y^sy^i\frac{\partial}{\partial p_k}=0.
\end{split}
\end{equation*}
\end{proof}
It is time to introduce one more notation. For any $\w\in T_\x M$ there is a unique co-vector $\w^v$ defined by $\w^v=g_\tv(\w,\cdot)$.

The homogeneity on the fibres of $F^*$ also implies
\begin{lemma}\label{minus}
$[\nabla^{\tv}_{\tv},\tv^v]=-\nabla^{\tv}_{\tv},$
\end{lemma}
\begin{proof}
Using Lemma \ref{HamLif} we obtain
$$\nabla^{\tv}_{\tv}=g^*_{ij}p_i\frac{\partial}{\partial x^j}-\frac{1}{2}\frac{\partial g^*_{ij}}{\partial x^k}p_ip_j\frac{\partial}{\partial p_k}.$$
Combining this identity with the identity $\tv^v=p_i\frac{\partial}{\partial p_i}$, we have
$$[\nabla^{\tv}_{\tv},\tv^v]+\nabla^{\tv}_{\tv}=\frac{1}{2}\frac{\partial^2 g^*_{ij}}{\partial p_a\partial x^k}p_a p_ip_j\frac{\partial}{\partial p_k} -\frac{\partial g^*_{ij}}{\partial p_a}p_a p_i\frac{\partial}{\partial x^j},$$
which vanishes by  Theorem \ref{homo}.
\end{proof}

The following lemma is useful when dealing the calculation of $\vec U$.
\begin{lemma}\label{CalW}
Let $U$ be the potential and $W_1,W_2$ be  vertical vector fields on $TM$, i.e. $\pi_*W_1=\pi_*W_2=0$.
Then
\begin{enumerate}
\item $\vec U=-(\nabla_\tv U)^v,$

\smallskip
\item $[\vec U,\tv^v]=\vec U,$

\smallskip
\item $g_{\tv}(W_1,W_2)=-\sigma(W_1^v,\nabla^\tv_{W_2}).$
\end{enumerate}
\end{lemma}

\begin{proof} The first and second assertions are verified by straightforward computations.
\begin{equation*}
\begin{split}
(\nabla_\tv U)^v&=(\frac{\partial U}{\partial x^i}g^*_{ij}\frac{\partial}{\partial x^j})^v=\frac{\partial U}{\partial x^i}g^*_{ij}g_{jk}\frac{\partial}{\partial p_k}=\frac{\partial U}{\partial x^i}\frac{\partial}{\partial p_i}=-\vec U.\\
 [\vec U,\tv^v]&=[-\frac{\partial U}{\partial x^i}\frac{\partial}{\partial p_i},p_j\frac{\partial}{\partial p_j}]=-\frac{\partial U}{\partial x^i}\frac{\partial}{\partial p_i}=\vec U.
\end{split}
\end{equation*}

For the third assertion, by linearity we can assume $W_1=\frac{\partial}{\partial x^i},\ W_2=\frac{\partial}{\partial x^j}$. Then
$$W_1^v=g_{ik}\frac{\partial}{\partial p_k},\ \nabla^{\tv}_{W_2}=\frac{\partial}{\partial x^j}+\Gamma^k_{jl}p_k\frac{\partial}{\partial p_l}.$$
Hence,
$$g_{\tv}(W_1,W_2)=g_{ij}=-\sigma(W_1^v,\nabla^{\tv}_{W_2}).$$

\end{proof}

The following two lemmas provide a coordinate-free method when calculating the curvatures.
\begin{lemma}\label{main1}
It holds the following identities.
\begin{equation*}
\begin{split}
~&\sigma([\nabla^{\tv}_{\tv},\nabla^{\tv}_{\w_1}],\nabla^{\tv}_{\w_2})=-g_{\tv}(\mathcal R(\tv,\w_1)\w_2,\tv)\\
&=g_{\tv}(\mathcal R(\tv,\w_1)\tv,\w_2)=-g_{\tv}(\mathcal R(\w_1,\tv)\tv,\w_2).
\end{split}
\end{equation*}
\end{lemma}
\begin{proof}
First of all, we remark that the last identity is from the antisymmetry of the Riemannian curvature in Finsler geometry w.r.t. the vectors $\tv,\w_1$ while
the second identity doesn't follow from the antisymmetry of the Riemannian curvature w.r.t. the vectors $\tv,\w_2$. Actually the latter antisymmetry  in general doesn't hold in Finsler geometry in contrast with the Riemannian geometry case. Rather, we can show that the sum
$g_{\tv}(\mathcal R(\tv,\w_1)\tv,\w_2)+g_{\tv}(\mathcal R(\tv,\w_1)\w_2,\tv)$
\smallskip
has a factor $C_{jbk}y^b$ (see, for instance, identity (3.4.4) in \cite{BCS}). This factor vanishes from \eqref{cartan}.

\medskip
It remains to prove the first identity.
Since both sides are linear in $\w_1$ and $\w_2$, it suffices to show the identity for $\w_1=\frac{\partial}{\partial{x^i}},\ \w_2=\frac{\partial}{\partial{x^j}}$.

\smallskip
The horizontal  lift w.r.t. Chern connection is Lagrangian(equivalent to the torsion freeness of Chern connection), i.e.
$$\sigma(\nabla^{\tv}_{\tv_1},\nabla^{\tv}_{\tv_2})=0,\, \forall \tv_1,\tv_2\in TM.$$
It follows
\begin{equation}\label{cal1}
\begin{split}
&~\sigma([\nabla^{\tv}_{\tv},\nabla^{\tv}_{\frac{\partial}{\partial x^i}}],\nabla^{\tv}_{\frac{\partial}{\partial x^j}})\\
&=\sigma\left(\left[y^a(\frac{\partial}{\partial x^a}+\Gamma^c_{ab}p_c\frac{\partial}{\partial p_b}),
\frac{\partial}{\partial x^i}+\Gamma^k_{il}p_k\frac{\partial}{\partial p_l}\right],\nabla^{\tv}_{\frac{\partial}{\partial x^j}}\right)\\
&=y^a\sigma\left(\left[\frac{\partial}{\partial x^a}+\Gamma^c_{ab}p_c\frac{\partial}{\partial p_b},
\frac{\partial}{\partial x^i}+\Gamma^k_{il}p_k\frac{\partial}{\partial p_l}\right],\nabla^{\tv}_{\frac{\partial}{\partial x^j}}\right).
\end{split}
\end{equation}

Recall that $y^a=g^*_{ab}p_b$, which gives
$$\frac{\partial{y^a}}{\partial{x^c}}=\frac{\partial{g^*_{ab}}}{\partial x^c}p_b.$$
And note that $\Gamma_{ij}^k$ depends not only on $\x$ but also on $\tv$.
Now we can do the following straightforward calculations.
\begin{eqnarray*}
&~&\left[\frac{\partial}{\partial x^a}+\Gamma^c_{ab}p_c\frac{\partial}{\partial p_b},
\frac{\partial}{\partial x^i}+\Gamma^k_{il}p_k\frac{\partial}{\partial p_l}\right]\\
&=&\frac{\partial \Gamma^k_{il}}{\partial x^a}p_k\frac{\partial}{\partial p_l}+\frac{\partial \Gamma^k_{il}}{\partial y^b}\frac{\partial{g^*_{bc}}}{\partial x^a}p_cp_k\frac{\partial}{\partial p_l}+\frac{\partial \Gamma^k_{il}}{\partial p_b}\Gamma^c_{ab}p_cp_k\frac{\partial}{\partial p_l}\\
&+&\Gamma^c_{ab}\Gamma^b_{il}p_c\frac{\partial}{\partial p_l}-\frac{\partial \Gamma^c_{ab}}{\partial x^i}p_c\frac{\partial}{\partial p_b}-\frac{\partial \Gamma^c_{ab}}{\partial y^l}\frac{\partial{g^*_{lk}}}{\partial x^i}p_kp_c\frac{\partial}{\partial p_b}\\
&-&\frac{\partial \Gamma^c_{ab}}{\partial p_l}\Gamma^k_{il}p_kp_c\frac{\partial}{\partial p_b}-\Gamma^l_{ab}\Gamma^k_{il}p_k\frac{\partial}{\partial p_b}.
\end{eqnarray*}
Then also straightforward calculation again gives
\begin{equation}\label{cal2}
\begin{split}
&~\sigma\left(\left[\frac{\partial}{\partial x^a}+\Gamma^c_{ab}p_c\frac{\partial}{\partial p_b},
\frac{\partial}{\partial x^i}+\Gamma^k_{il}p_k\frac{\partial}{\partial p_l}\right],\nabla^{\tv}_{\frac{\partial}{\partial x^j}}\right)\\
&=-\frac{\partial \Gamma^k_{ij}}{\partial x^a}p_k-\frac{\partial\Gamma^k_{ij}}{\partial y^b}\frac{\partial{g^*_{bc}}}{\partial x^a}p_cp_k-\frac{\partial\Gamma^k_{ij}}{\partial p_b}\Gamma^c_{ab}p_cp_k\\
&\quad-\Gamma^c_{ab}\Gamma^b_{ij}p_c+\frac{\partial\Gamma^c_{aj}}{\partial x^i}p_c+\frac{\partial\Gamma^c_{aj}}{\partial y^l}\frac{\partial{g^*_{lk}}}{\partial x^i}p_cp_k\\
&\quad+\frac{\partial\Gamma^c_{aj}}{\partial p_l}\Gamma^k_{il}p_kp_c+\Gamma^l_{aj}\Gamma^k_{il}p_k.
\end{split}
\end{equation}
On the other hand, from \eqref{RieCoe} and \eqref{N_Gamma} we have
\begin{equation*}
\begin{split}
&~g_{\tv}(\mathcal R(\tv,\frac{\partial}{\partial x^i})\frac{\partial}{\partial x^j},\tv)=\\
&y^ap_b\left(\frac{\partial \Gamma^b_{ji}}{\partial x^a}-\frac{\partial \Gamma^b_{ja}}{\partial x^i}+\frac{\partial \Gamma^b_{ja}}{\partial y^t}\Gamma_{ik}^ty^k-\frac{\partial \Gamma^b_{ji}}{\partial y^t}\Gamma_{ak}^ty^k+\Gamma_{as}^b\Gamma_{ji}^s-\Gamma_{ja}^s\Gamma_{is}^b\right).
\end{split}
\end{equation*}
Using this identity and \eqref{cal1},\eqref{cal2} we obtain
\begin{equation}\label{cal3}
\begin{split}
&~\sigma([\nabla^{\tv}_{\tv},\nabla^{\tv}_{\frac{\partial}{\partial x^i}}],\nabla^{\tv}_{\frac{\partial}{\partial x^j}})+g_{\tv}\left(\mathcal R(\frac{\partial}{\partial x^j},\tv)\tv,\frac{\partial}{\partial x^i}\right)\\
&=-\frac{\partial \Gamma^k_{ij}}{\partial y^b}\frac{\partial{g^*_{bc}}}{\partial x^a}p_cp_ky^a-\frac{\partial \Gamma^k_{ij}}{\partial p_b}\Gamma^c_{ab}p_cp_ky^a+\frac{\partial \Gamma^c_{aj}}{\partial y^l}p_kp_cy^a\frac{\partial{g^*_{lk}}}{\partial x^i}\\
&+\frac{\partial \Gamma^c_{aj}}{\partial p_l}\Gamma^k_{il}p_kp_cy^a-\frac{\partial \Gamma^b_{ji}}{\partial y^t}\Gamma_{ak}^tp_b y^a y^k
+\frac{\partial \Gamma^b_{ja}}{\partial y^t}\Gamma_{ik}^tp_b y^a y^k
\end{split}
\end{equation}

The last goal is to show the right-hand side of the last identity vanishes. Indeed, since Chern connection is almost compatible with the Finsler metric(see e.g. \cite{Shen}),
\begin{equation}\label{cal4}
\begin{split}
&~\frac{\partial{g^*_{bc}}}{\partial x^a}p_c=\frac{\partial{g^*_{bc}}}{\partial x^a}g_{cs}y^s=-\frac{\partial{g_{cs}}}{\partial x^a}g^*_{bc}y^s\\
&=-g^*_{bc}y^s\left(\Gamma_{ac}^tg_{ts}+\Gamma_{as}^tg_{tc}+C_{cts}\Gamma_{ac}^ty^c\right)\\
&=-\Gamma_{ac}^tg^*_{bc}p_t-\Gamma_{as}^by^s.
\end{split}
\end{equation}
Similarly,
\begin{equation}\label{cal5}
\frac{\partial{g^*_{lk}}}{\partial x^i}p_k=-\Gamma_{ik}^tg^*_{lk}p_t-\Gamma_{is}^ly^s.
\end{equation}

Using \eqref{cal4},\eqref{cal5} we finally verify by a straightforward calculation that the right-hand side of \eqref{cal3} vanishes, as claimed.
\end{proof}

\begin{lemma}\label{main2}
Denote by $\mathbf{Hess}_{\tv}$ the Hessian w.r.t. the Riemannian metric $g_\tv$ and by $\mathbf{P}$ the Chern curvature.
\begin{equation*}
\sigma([\vec
U,\nabla^{\tv}_{\w_1}],\nabla^{\tv}_{\w_2})=-{\mathbf{
Hess_{\tv}}}\ U(\w_1,\w_2)-g_{\tv}\left(\mathbf{P_{\tv}}\left(\w_1,\nabla_\tv U,\w_2\right),\tv\right).
\end{equation*}
\end{lemma}

\begin{proof}
Again, it suffices to prove the case when $\w_1=\frac{\partial}{\partial x^i},\w_2=\frac{\partial}{\partial x^j}$. A straightforward calculation shows
\begin{equation}\label{cal6}
\sigma\left([\vec
U,\nabla^{\tv}_{\frac{\partial}{\partial x^i}}],\nabla^{\tv}_{\frac{\partial}{\partial x^j}}\right)=-\frac{\partial^2U}{\partial x^i\partial x^j}+\frac{\partial U}{\partial x^k}\Gamma_{ij}^k+\frac{\partial \Gamma_{ij}^l}{\partial p_k}\frac{\partial U}{\partial x^k}p_l.
\end{equation}

On the other hand, we have
\begin{equation}\label{cal7}
\mathbf{Hess}_{\tv}\ U\left(\frac{\partial}{\partial x^i},\frac{\partial}{\partial x^j}\right)=\frac{\partial^2U}{\partial x^i\partial x^j}-\frac{\partial U}{\partial x^k}\Gamma_{ij}^k.
\end{equation}

Next, since $\nabla_\tv U=\frac{\partial U}{\partial x^a}g^*_{ab}\frac{\partial}{\partial x^b}$, it follows (see e.g. \cite{Shen}) that
\begin{equation*}
g_{\tv}\left(\mathbf{P_{\tv}}\left(\frac{\partial}{\partial x^i},\nabla_\tv U,\frac{\partial}{\partial x^j}\right),\tv\right)=-\frac{\partial \Gamma_{ji}^l}{\partial y^k}y^bg^*_{ak}g_{lb}\frac{\partial U}{\partial x^a}=-\frac{\partial \Gamma_{ji}^l}{\partial p_k}p_l\frac{\partial U}{\partial x^k}.
\end{equation*}

Combining this identity with \eqref{cal6},\eqref{cal7} we complete the proof of the lemma.
\end{proof}

\subsection{Canonical complements}
 Recall that
there is a canonical splitting
  \begin{equation}\label{splitting1}
 W_{\lambda}=\Pi_\lambda\oplus \mathfrak J^{trans}_\lambda,
  \end{equation}
   where 
   $\mathfrak J^{trans}_\lambda$ is the canonical complement.
A similar situation happens for the non-reduced case. Namely, it holds the following  canonical splitting
 \begin{equation}\label{splitting2}
 T_\lambda T^* M=T_\lambda T^*_\x M\oplus\bar{\mathfrak J}_\lambda^{trans},
  \end{equation}
   where 
   $\bar {\mathfrak J}^{trans}\lambda$ is the canonical complement.
In other words, $\mathfrak J^{trans}_\lambda,\bar{\mathfrak J}^{trans}_\lambda$ are actually
(nonlinear) Ehresmann connections of $\Pi_\lambda$ in
$W_{\lambda}$ and of $T_\lambda T^*_\x M$ in $T_\lambda T^*M$, respectively.

\medskip
For the non-reduced case, a completely similar argument as in the proof of Lemma 11.1 in \cite{Le1} gives the following
\begin{lemma}\label{splitting3}
$\bar{\mathfrak J}^{trans}(\lambda)$ coincides with the Chern connection with the reference vector $\tv$, i.e. $\bar{\mathfrak J}^{trans}_\lambda={\rm span}\left\{\nabla^{\tv}_{\frac{\partial}{\partial x^i}},i=1,...,n\right\}.$
\end{lemma}

For the canonical complement $\mathfrak J^{trans}_\lambda$, we have the following
\begin{lemma}\label{splitting4}$\bar{\mathfrak J}^{trans}_\lambda={\rm span}\left\{\nabla^{\tv}_{\xi^h}-\frac{1}{F(\tv)^2}\cdot g_\tv(\xi^h,\nabla_\tv U)\tv^v,\xi\in\Pi_\lambda\right\}.$
\end{lemma}
\begin{proof}
Assume $$\bar{\mathfrak J}^{trans}_\lambda={\rm span}\{\nabla^{\tv}_{\xi^h}+A(\xi)\tv^v,\xi\in\Pi_\lambda\}.$$
Note that $\sigma(\vec \ham,\tv^v)=g_{\tv}(\tv,\tv)=F(\tv)^2$. Hence, from
the fact that $\bar{\mathfrak J}^{trans}_\lambda$ is tangent to the Hamiltonian
vector field $\vec \ham$ and  Lemma \ref{CalW}, we get easily that
$$A(\tv)=-\frac{1}{F(\tv)^2}\cdot g_\tv\left(\xi^h,\nabla_\tv U\right),$$which completes the proof of the lemma.
\end{proof}

\subsection{Calculations of the curvatures}
It is convenient to introduce the notation of parallel transport of a vector field along the Hamiltonian flows.

Let $\lambda\in T^*M$ and let $\lambda(t)=e^{t\vec \ham}\lambda$.
Assume that $(E^\lambda(t),F^\lambda(t))$ is a normal moving frame
of the Jacobi curve $\mathfrak J_\lambda(t)$ attached at point
$\lambda$. Let $\mathfrak E$ be the Euler field on $T^*M$, i.e. the
infinitesimal generator of the homotheties  on its fibers. Clearly,
$$T_\lambda(T^*M)=T_\lambda\mathcal
H_{c}\oplus\mathbb R \mathfrak E(\lambda).$$

 The flow $e^{t\vec
\ham}$ on $T^*M$ induces the push-forward maps $e^{t\vec \ham}_*$ between
the corresponding tangent spaces $T_\lambda T^*M$ and
$T_{\lambda(t)}T^*M$, which in turn induce naturally the maps
between the spaces $T_\lambda(T^*M)/{\rm span}\{\vec \ham(\lambda)
\}$ and $T_{\lambda(t)}T^*M/{\rm
span}\{\vec \ham(\lambda(t))\}$.
The map $\mathcal K^t$ between $T_\lambda(T^*M)/{\rm span}\{\vec
\ham(\lambda)\}$ and
$T_{\lambda(t)}T^*M/{\rm span}\{\vec \ham(\lambda(t))\}$, sending $E^\lambda(0)$ to
$e^{t\vec \ham}_*E^{\lambda}(t)$, $F^\lambda(0)$ to $e^{t\vec
\ham}_*F^{\lambda}(t)$, and the equivalence class of $\mathfrak
E(\lambda)$ to the
equivalence class of $\mathfrak E(e^{t\vec \ham}\lambda)$, is independent of the
choice of normal moving frames. The map $\mathcal K^t$ is called
\emph{the parallel transport} along the extremal $e^{t\vec
\ham}\lambda$ at time $t$.  For any $v\in T_\lambda(T^*M)/{\rm
span}\{\vec \ham(\lambda)\}$, its
image $v(t)=\mathcal K^t(v)$ is called \emph{ the parallel transport
of $v$ at time $t$}.

 Note that from the definition of the non-reduced
Jacobi curves and the construction  of normal moving frames it
follows that the restriction of the parallel transport $\mathcal
K^t$ to the vertical subspace $T_\lambda(T_\x^*M)$ of
$T_\lambda(T^*M)$ can be considered as a map onto the vertical
subspace $T_{\lambda(t)}(T_{\pi(\lambda(t))}^*M)$ of
$T_{\lambda(t)}(T^*M)$. A vertical vector field $V$ is called
\emph{parallel} if $V(e^{t\vec \ham}\lambda)=\mathcal
K^t\bigl(V(\lambda)\bigr)$.

%
%

\medskip
The rest of the draft is devoted to the proof of Theorems \ref{CurFin} and \ref{FinMec2}.
\begin{proof} [Proof of Theorem \ref{CurFin}] Let $\bar W_1,\bar W_2$ be parallel vertical vector fields such
that $\bar W_1(\lambda)=\bar\xi, \bar W_2(\lambda)=\bar\eta$. From Lemma \ref{splitting3} the lift of $W_i$ in ${\mathfrak J}_\lambda^{trans}$ is $\nabla^{\tv}_{W_i^h}$.
Then it follows from  Lemma \ref{CalW} that
\begin{equation}\label{preli}
g_\tv(({\bar{\mathfrak R}}_\lambda\xi)^h,\eta^h)=-\sigma(\bar{\mathfrak R}_\lambda\xi,\nabla^{\tv}_{\eta^h})=-\sigma\left([\vec{\ham},\nabla^{\tv}_{W_1^h}],\nabla^{\tv}_{W_2^h}\right).
\end{equation}
Combining this with Lemma \ref{HamLif}, Lemma \ref{main1} and Lemma \ref{main2}, we complete the proof of the present theorem.
\end{proof}

\begin{proof}[Proof of Theorem \ref{FinMec2}]
Let $W_1,W_2$ be  vertical parallel vector fields such that $W_1(\lambda)=\xi, W_2(\lambda)=\eta$. From Lemma \ref{HamLif}, Lemma \ref{CalW} and Lemma \ref{splitting4} we have
\begin{equation*}
\begin{split}
&g_\tv(({\mathfrak R}_\lambda\xi)^h,\eta^h)\\
=&-\sigma\left(\left[\nabla_\tv^\tv+\vec U,\nabla^{\tv}_{W_1^h}
-\frac{ g_\tv(W_1^h,\nabla_\tv U)}{F(\tv)^2}\tv^v\right],\nabla^{\tv}_{W_2^h}-\frac{ g_\tv(W_2^h,\nabla_\tv U)}{F(\tv)^2}\tv^v\right).
\end{split}
\end{equation*}

Next we write  the right-hand side of the last identity as the sum of the following $T_is$ and calculate each $T_i$ in turn. Note that when dealing with $T_5$ and $T_6$ we used the identity
$$\sigma\left(\tv^v,\nabla^\tv_{W_2^h}\right)=-g_{\tv}(\tv,W_2^h)=0.$$
\begin{equation*}
\begin{split}
T_1&=-\sigma\left([\nabla_{\tv}^{\tv},\nabla^{\tv}_{W_1^h}],\nabla^{\tv}_{W_2^h}\right),\\
T_2&=\frac{g_\tv(W_2^h,\nabla_\tv U)}{F(\tv)^2}\sigma\left([\nabla_{\tv}^{\tv},\nabla^{\tv}_{W_1^h}],\tv^v\right),\\
T_3&=-\sigma\left([\vec U,\nabla^{\tv}_{W_1^h}],\nabla^{\tv}_{W_2^h}\right),\\
T_4&=\frac{g_\tv(W_2^h,\nabla_\tv U)}{F(\tv)^2}\sigma\left([\vec U,\nabla^{\tv}_{W_1^h}],\tv^v\right),\\
T_5&=\frac{g_\tv(W_1^h,\nabla_\tv U)}{F(\tv)^2}\sigma\left([\nabla_{\tv}^{\tv},\tv^v],\nabla^{\tv}_{W_2^h}\right),\\
T_6&=\frac{g_\tv(W_1^h,\nabla_\tv U)}{F(\tv)^2}\sigma\left(\left[\vec U,\tv^v\right],\nabla^{\tv}_{W_2^h}\right),\\
T_7&=-\frac{g_\tv(W_1^h,\nabla_\tv U)}{F(\tv)^2}\cdot\frac{g_\tv(W_2^h,\nabla_\tv U)}{F(\tv)^2}\sigma([\nabla_{\tv}^{\tv},\tv^v],\tv^v).
\end{split}
\end{equation*}
%

\medskip
First of all, from Lemma \ref{main1} we have
\begin{equation}\label{T1}
T_1=-g_\tv\left(\mathcal R(W_2^h,\tv)\tv,W_1^h\right).
\end{equation}

\medskip
For $T_2$, we use that the symplectic form $\sigma$ is closed to show
\begin{equation*}
\begin{split}
0&=d\sigma(\nabla_{\tv}^{\tv},\nabla^{\tv}_{W_1^h},\tv^v)\\
&=\nabla_{\tv}^{\tv}(\sigma(\nabla^{\tv}_{W_1^h},\tv^v))-\nabla^{\tv}_{W_1^h}(\sigma(\nabla_{\tv}^{\tv},\tv^v))\\
&+\tv^v(\sigma(\nabla_{\tv}^{\tv},\nabla_{W_1^h}^{\tv}))-\sigma([\nabla_{\tv}^{\tv},\nabla^{\tv}_{W_1^h}],\tv^v)\\
&+\sigma([\nabla_{\tv}^{\tv},\tv^v],\nabla^{\tv}_{W_1^h})-\sigma([\nabla^{\tv}_{W_1^h},\tv^v],\nabla_{\tv}^{\tv}).
\end{split}
\end{equation*}

Then we make the following calculations. Using the last item in Lemma \ref{CalW} we have
\begin{equation*}
\sigma(\nabla^{\tv}_{W_1^h},\tv^v)=g_\tv(W_1^h,\tv)=0
\end{equation*}
and
\begin{equation*}
\begin{split}
\nabla^{\tv}_{W_1^h}(\sigma(\nabla_{\tv}^{\tv},\tv^v))&=\nabla^{\tv}_{W_1^h}(g_\tv(\tv,\tv))=\nabla^{\tv}_{W_1^h}(F(\tv)^2)\\
&=\sigma(2\vec\ham-2\vec U,\nabla^{\tv}_{W_1^h})=2\sigma(\nabla^{\tv}_{\tv},\nabla^{\tv}_{W_1^h})=0.
\end{split}
\end{equation*}
Note that in the last equality we used Lemma \ref{hamil}, Lemma \ref{HamLif} and the fact  that Chern connection is torsion free.
Next we have
$$\tv^v(\sigma(\nabla_{\tv}^{\tv},\nabla_{W_1^h}^{\tv}))=\tv^v(0)=0.$$
Similarly, we apply Lemma \ref{minus} to get
$$\sigma([\nabla_{\tv}^{\tv},\tv^v],\nabla^{\tv}_{W_1^h})=-\sigma(\nabla_{\tv}^{\tv},\nabla^{\tv}_{W_1^h})=0.$$
And we also have
\begin{equation*}
\begin{split}
\sigma([\nabla^{\tv}_{W_1^h},\tv^v],\nabla_{\tv}^{\tv})&=-[\nabla^{\tv}_{W_1^h},(\tv)^v](\ham-U)\\
&=-\nabla^{\tv}_{W_1^h}(\tv^v(\ham-U))+\tv^v(\nabla^{\tv}_{W_1^h}(\ham-U))\\
&=-\nabla^{\tv}_{W_1^h}(2(\ham-U))+\tv^v(\sigma(\vec\ham-\vec U,\nabla^{\tv}_{W_1^h}))\\
&=-2\sigma(\nabla_{\tv}^{\tv},\nabla^{\tv}_{W_1^h})+\tv^v(\sigma(\nabla_{\tv}^{\tv},\nabla^{\tv}_{W_1^h}))=0.
\end{split}
\end{equation*}
Summarizing above we have
\begin{equation}\label{T2}
T_2=0.
\end{equation}

\smallskip
For $T_3$, from Lemma \ref{main2} we have
\begin{equation*}\label{T3}
T_3={\mathbf{
Hess_{\tv}}}\ U(W_1^h,W_2^h)+g_\tv\left(\mathbf{P_{\tv}}\left(W_1^h,\nabla_\tv U,W_2^h\right),\tv\right).
\end{equation*}

\smallskip
The case of $T_4$ is similar to that of $T_2$. we first of all have
\begin{equation*}
\begin{split}
0&=d\sigma(\vec U,\nabla^{\tv}_{W_1^h},\tv^v)\\
&=\tv^v(\sigma(\vec U,\nabla_{\tv}^{\tv}))-\sigma([\vec U,\nabla^{\tv}_{W_1^h}],\tv^v)\\
&+\sigma([\vec U,\tv^v],\nabla^{\tv}_{W_1^h})-\sigma([\nabla^{\tv}_{W_1^h},\tv^v],\vec U).
\end{split}
\end{equation*}
Then we make the following calculations. Using  Lemma \ref{CalW} we have
\begin{equation*}
\sigma([\vec U,\tv^v],\nabla^{\tv}_{W_1^h})=g_{\tv}(\nabla_\tv U,W_1^h)
\end{equation*}
and
\begin{equation*}
\begin{split}
\sigma([\nabla^{\tv}_{W_1^h},\tv^v],\vec U)&=-[\nabla^{\tv}_{W_1^h},\tv^v](U)\\
&=\tv^v(\nabla^{\tv}_{W_1^h}(U))\\
&=\tv^v(\sigma(\vec U,\nabla^{\tv}_{W_1^h})).
\end{split}
\end{equation*}
And using  Lemma \ref{CalW} again we obtain
$$\sigma([\vec U,\tv^v],\nabla^{\tv}_{W_1^h})=\sigma(\vec U,\nabla^{\tv}_{W_1^h})=g_{\tv}(\nabla_\tv U,W_1^h).$$
Summarizing, we have
\begin{equation}\label{T4}
T_4=\frac{1}{F(v)^2}g_{\tv}(W_1^h,\nabla_\tv U)g_{\tv}(W_2^h,\nabla_\tv U).
\end{equation}

\medskip
Using Lemma \ref{minus} and the fact that the Chern connection is torsion free,  we get
\begin{equation}\label{T}
\sigma([\nabla_{\tv}^{\tv},\tv^v],\nabla^{\tv}_{W_2^h})=-\sigma(\nabla_{\tv}^{\tv},\nabla^{\tv}_{W_2^h})=0.
\end{equation}
Hence $T_5=0.$

\smallskip
For $T_6$, since
$$\sigma([\vec U,\tv^v],\nabla^{\tv}_{W_2^h})=\sigma(\vec U,\nabla^{\tv}_{W_2^h})=g_{\tv}(W_2^h,\nabla_\tv\ U),$$
then
\begin{equation}
T_6=\frac{1}{F(\tv)^2}g_\tv(W_1^h,\nabla_\tv U)g_\tv(W_2^h,\nabla_\tv U).
\end{equation}
%

\smallskip
For $T_7$, we make the following calculations.
$$\sigma([\nabla_{\tv}^{\tv},\tv^v],\tv^v)=-\sigma(\nabla_{\tv}^{\tv},\tv^v)=-F(\tv)^2.$$
This gives
\begin{equation}\label{T7}
T_7=\frac{1}{F(\tv)^2}g_\tv(W_1^h,\nabla_\tv U)g_\tv(W_2^h,\nabla_\tv W).
\end{equation}

Summarizing the above calculations for  $T_i(1\leq i\leq 7)$, we complete the proof of the theorem.

\end{proof}


%



\medskip

\subsection*{Acknowledgement.}
The author would like to acknowledge Professor Zhiguang Hu for stimulating discussions on Finsler geometry.


\begin{thebibliography}{200}
\bibitem{AgBoBa} A. Agrachev, U.Boscain and D. Barilari:  A Comprehensive Introduction to Sub-Riemannian Geometry,
Cambridge Studies in Advanced Mathematics, Cambridge University Press, 2019.

\bibitem{ABR1} A.Agrachev, D.Barilari and L.Rizzi: Sub-Riemannian curvature in contact geometry, J. Geom. Anal. 27(2017), 366--408.

\bibitem{ABR2} A.Agrachev, D.Barilari and L.Rizzi: Curvature: a variational approach,
Memoirs of the AMS, 256 (2018), no. 1225, vi + 142 pp.


\bibitem{AgGa} A. A. Agrachev and R. V. Gamkrelidze: Feedback-invariant optimal control theory - I. Regular extremals,
J. Dynamical and Control Systems,  3 (1997), 343-389.

\bibitem{AgChZe} A. Agrachev, N. Chtcherbakova, I. Zelenko:
On Curvatures and Focal Points of Dynamical {L}agrangian Distributions and their Reductions by First
Integrals, J. Dynamical and Control Systems, 11 (2005), 297-327.

\bibitem{AgCh} A. Agrachev, N. Chtcherbakova: Hamiltonian Systems of Negative Curvature are Hyperbolic,Russian Math. Dokl., 400(2005), 295-298.

\bibitem{AgZe1} A. Agrachev, I. Zelenko: Geometry of Jacobi curves I, J. Dynamical and Control Systems, V. 8, No.1 (2002), 93-140.


\bibitem{Ano} D. V. Anosov: Geodesic flows on the closed {R}iemannian manifold of negative curvature, Proceedings of the Steklov Institute of Mathematics, AMS, Providence, RI, 90(1967), 3-209.


\bibitem{BCS} D. Bao, S. S. Chern and Z. Shen: An introduction to Riemann-Finsler geometry, Springer-Verlag, New York, 2000.

\bibitem{Fou} P. Foulon: Estimation de l$'$entropie des syst$\grave{e}$mes lagrangiens sans points conjugu\'es, Ann. Inst. Henri Poincar\'e 57(2) (1992), 117-146.

\bibitem{KaHa} A. Katok, B. Hasselblatt: Introduction to the Modern Theory of Dynamical Systems, Encyclopedia of Mathematics and its Applications(54), Cambridge Univ. Press, 1995.


 \bibitem{Le1} P.W.Y. Lee: Displacement interpolations from a Hamiltonian point of view, J. Func. Anal., 265(12): 3163-3203, 2013.


\bibitem{Li2} Chengbo Li: A note on hyperbolic flow in sub-Riemannian structure with transverse symmetries, Acta.Appl.Math., 117(1): 71-91,2012.


 \bibitem{Li1}  C. Li: On curvature-type invariants for natural mechanical systems on sub-Riemannian structures associated with a principle G-bundle, in the book `` Geometric Control Theory and sub-Riemannain Geometry'', G. Stefani, U. Boscain, M. Sigalotti, J.-P. Gauthier, A. Sarychev (Eds.), Springer INdAM series, Vol. 7, to appear in 2014.

\bibitem{LiZe1} C. Li, I. Zelenko: Differential geometry of curves in Lagrange Grassmannians with given Young diagram. Differential Geom. Appl. 27 (2009), no. 6, 723--742.

\bibitem{LiZe2} C. Li, I. Zelenko: Parametrized curves in Lagrange Grassmannians, C.R. Acad. Sci. Paris, Ser. I, Vol. 345, Issue 11, 647--652.

\bibitem{LiZe3} C. Li, I.Zelenko: Jacobi Equations and Comparison Theorems for Corank 1 sub-Riemannian Structures with Symmetries, Journal of Geometry and Physics 61 (2011) 781--807.

\bibitem{Mo} R. Montgomery: A tour of subriemannian geometries, their geodesics and applications. Mathematical Surveys and Monographs, 91. American Mathematical Society, Providence, RI, 2002.





\bibitem{Pon} L. S. Pontryagin and V. G. Boltyanskii and R. V. Gamkrelidze and E. F. Mischenko:
The Mathematical Theory of Optimal Processes, Wiley, New York,1962.


\bibitem{Shen} Z. Shen: Lectures on Finsler Geometry, World Scientific, Singapore, 2001.


\end{thebibliography}
\end{document}